\documentclass[11pt]{article}
\usepackage[latin1]{inputenc} 
\usepackage[T1]{fontenc}
\usepackage[english]{babel}
\usepackage{lmodern}
\usepackage[babel=true,kerning=true]{microtype}
\usepackage{amsmath,amssymb}
\usepackage{amsfonts}
\usepackage{amsthm}
\usepackage{mathrsfs}
\usepackage{calc}
\usepackage{appendix}
\usepackage{graphicx}
\usepackage{pgf,tikz}
\usetikzlibrary{arrows}
\usepackage{vmargin}
\usepackage{enumitem}
\usepackage{variations}
\usepackage{stmaryrd}
\usepackage{tkz-tab}
\usepackage[normalem]{ulem}
\usepackage{ulem}
\usepackage{eurosym}
\usepackage{fancyhdr}
\usepackage{latexsym}
\usepackage{bbding}
\usepackage{array}
\usepackage{pstricks,pst-all,pstricks-add,pst-tree,pst-3dplot}
\usepackage[tikz]{bclogo}
\usetikzlibrary{shadows}
\usepackage[framemethod=tikz]{mdframed}
\usepackage[sectionbib]{bibunits}

\theoremstyle{plain}
\theoremstyle{definition}
\newtheorem{theo}{Theorem}
\newtheorem{defi}[theo]{Definition} 
\newtheorem{prop}[theo]{Proposition}

\newtheorem{lem}[theo]{Lemma}

\theoremstyle{remark}

\newcommand{\ensemblenombre }[1]{\mathbb{#1}}
\newcommand{\N}{\ensemblenombre{N}}
\newcommand{\Z}{\ensemblenombre{Z}}

\newcommand{\C}{\ensemblenombre{C}}

\newcommand{\abs}[1]{\left\lvert#1\right\rvert}

\title{\textbf{Roots of Chebyshev Polynomials: \\ a purely algebraic approach}}
\author{Lionel Ponton \\ \texttt{lionel.ponton@gmail.com}}
\date{}
\setmarginsrb{2cm}{2cm}{2cm}{2.5cm}{0cm}{0cm}{0cm}{1cm} 
\begin{document}
\maketitle

\begin{abstract}
	By using purely algebraic tools, we establish well-known properties of roots of Chebyshev polynomials. Especially, we show that these zeros are simple and lie in $(-1,1)$ and we prove in two ways that they are mostly irrational.
\end{abstract}

\medskip


\medskip


\section{Introduction}

The Chebyshev polynomials form one of the most famous and classical family of polynomials. They are related to many others subjects such as orthogonal polynomials, Jacobi polynomials, sine and cosine functions and approximation theory and may be define in many ways: recurrence relation, differential equation, trigonometric relation, generating function, etc.

In  approximation theory, the roots of these polynomials have a key role in connection with the Runge's phenomenon \cite[p. 143-144]{MH03}. According to the orthogonality of Chebyshev polynomials \cite[p. 16]{Sny66}, it is well-know that these zeros are simple and lie in the open interval $(-1,1)$ \cite[p. 7]{Sny66}. Moreover, the trigonometric relations satisfied by Chebyshev polynomials lead to some explicit expressions of these roots of the form $\cos(r\pi)$ where $r$ is a rational number \cite[p. 14]{Sny66}. Since the arithmetical properties of this kind of numbers are  well-know \cite[Th. 6.16, p. 308]{NZM91}, it is easy to derive that these roots are mostly irrational \cite[Th. 1]{Pon18}. 

We can note that, even if Chebyshev polynomials are algebraic objects, they are mainly study using calculus and trigonometry. In this paper, we adopt the following point of view: building on an algebraic definition of the Chebyshev polynomials, we derive all the above properties of their roots using only algebraic tools. More precisely, in section 2, we show that all the zeros of Chebyshev polynomials are real and lie in $(-1,1)$ by means of recursive sequences and very basic computation on complex numbers. Then, in section 3, we establish that these roots are simple using polynomial arithmetic. Finally, in section 4, on the basis of explicit expressions of Chebychev polynomials, we show that their zeros are mostly irrational in two ways: first, by using their expression directly; then, by means of shifted Chebychev polynomials. Note that, even if all the properties of Chebyshev polynomials stated and used in the following are already known, we prove each of them both to provide a self-contained article, and to ensure that it only use algebraic tools.

\bigskip

In that respect, we start with the following definition.

\begin{defi}
	The Chebyshev polynomials of first kind are defined by the recurrence relation:
	\[\begin{cases} T_0=1,~T_1=X, \\ \forall n\in\N,~T_{n+2}=2XT_{n+1}-T_n \end{cases} \]
	and the Chebyshev polynomials of second kind are defined by the recurrence relation:
	\[\begin{cases} U_0=1,~U_1=2X, \\ \forall n\in\N,~U_{n+2}=2XU_{n+1}-U_n \end{cases}\]
	
	Throughout this article, we designate $\N$ as the set of nonnegative integers, $\N^*$ as the set of positive integers, $\Z$ as the set of rational integers and $\C$ as the set of complex numbers.
	
	It is easy to check by double induction that, for all $n\in\N$, $T_n$ and $U_n$ are in $\Z[X]$, have degree $n$ and are even or odd according to $n$ is even or odd. 
\end{defi}

\section{The roots of Chebyshev polynomials lie in $(-1,1)$}

\begin{lem} For all $n\in\N$, $T_n(1)=1$ and $U_n(1)=n+1$.
\label{lem_value_in_1}
\end{lem}

\begin{proof} This follows from an obvious double induction on $n$.
\end{proof}

\begin{lem} \label{lem_linear_rec} Let $w\in\C\setminus\{-1,1\}$. We define two sequences $(t_n)$ and $(u_n)$ of complex numbers by the recurrence relations:
	\[\begin{cases} t_0=1,~t_1=w, \\ \forall n\in\N,~t_{n+2}=2wt_{n+1}-t_n \end{cases} \qquad \text{and} \qquad \begin{cases} u_0=1,~u_1=2w, \\ \forall n\in\N,~u_{n+2}=2wu_{n+1}-u_n \end{cases} \]
	Then, there exists a complex number $r$ such that $w=\frac{1}{2}(r+\frac{1}{r})$ and, for all $n\in\N$,
	\[t_n=\frac{r^{2n}+1}{2r^n} \qquad \text{and} \qquad u_n=\frac{r^{2n+2}-1}{r^n(r^2-1)}.\]
\end{lem}

\begin{proof} The characteristic polynomial of $(t_n)$ and $(u_n)$ is $P=X^2-2wX+1$. Let $r$ be a complex root of $P$. Since the constant term of $P$ is $1$, the other complex root of $P$ is $\frac{1}{r}$. Moreover, since $w\notin\{-1,1\}$, $r\neq \frac{1}{r}$. Thus, there exist complex numbers $\alpha$, $\beta$, $\gamma$ and $\delta$ such that, for all $n\in\N$, $t_n=\alpha r^n + \beta \frac{1}{r^n}$ and $u_n=\gamma r^n + \delta \frac{1}{r^n}$. Since the term of $X$ in $P$ is $-2w$, we have $w=\frac{1}{2}(r+\frac{1}{r})$ and then  we deduce from $t_0=1$ and $t_1=w$ that $\alpha=\beta=\frac{1}{2}$ and  from $u_0=1$ and $u_1=2w$ that $\gamma=\frac{r^2}{r^2-1}$ and $\delta=\frac{-1}{r^2-1}$ which yields the result.
\end{proof}

\begin{theo} \label{theo_roots_intervalle} Let $n\in\N^*$. All the complex roots of $T_n$ and $U_n$ are real and lie in $(-1,1)$.
\end{theo}

\begin{proof} Let $w$ be a complex root of $T_n$. Due to Lemma \ref{lem_value_in_1} and to the parity of $T_n$, $w\notin\{-1,1\}$. Put, for every $k\in\N$, $t_k:=T_k(w)$. Then, $t_0=1$, $t_1=w$ and, for all $k\in\N$, $t_{k+2}=2wt_{k+1}-t_k$. Thus, by Lemma \ref{lem_linear_rec}, there is a complex number $r$ such that $w=\frac{1}{2}(r+\frac{1}{r})$ and, for all $k\in\N$, $T_k(w)=\frac{r^{2k}+1}{2r^k}$. Since $T_n(w)=0$, it follows that $r^{2n}+1=0$ and, in particular, $\abs{r}=1$. Then, $\overline{r}=\frac{\abs{r}^2}{r}=\frac{1}{r}$ and thus $w=\frac{1}{2}(r+\overline{r})=\text{Re}(r)$ is a real number. Moreover, $\abs{w}=\abs{\text{Re}(r)} \leqslant \abs{r}=1$ and we conclude that $w\in(-1,1)$. The proof for the roots of $U_n$ is exactly the same.
\end{proof}

\section{The roots of Chebyshev polynomials are simple}

\begin{lem} \label{lem_coprime} For all $n\in\N$, $T_n$ and $T_{n+1}$ are coprime.
\end{lem}

\begin{proof} For all $n\in\N$, we put $D_n$ the (monic) gcd of $T_n$ and $T_{n+1}$. Let $n\in\N$. By definition, $D_n$ divides $T_n$ and $T_{n+1}$ thus, since $T_{n+2}=2XT_{n+1}-T_n$, $D_{n}$ divides $T_{n+2}$ too. It follows that $D_{n}$ divides $D_{n+1}$. Similarly, since $T_{n}=2XT_{n+1}-T_{n+2}$, $D_{n+1}$ divides $T_n$ and thus it divides $D_n$. Since $D_n$ and $D_{n+1}$ are monic, we conclude that $D_n=D_{n+1}$. Then, $(D_n)$ is constant and, for all $n\in\N$, $D_n=D_0=\text{gcd}(1,X)=1$. 
\end{proof}

\begin{lem} \label{lem_links_T_n_U_n} For all $n\in\N$,
	\begin{align}
		&T_{n+1}=U_{n+1}-XU_{n} \label{eq_1} \\
		&T_{n+1}'=(n+1)U_n \label{eq_2} \\
		&T_{n+2}=XT_{n+1}-(1-X^2)U_n \label{eq_3} \\
		&(1-X^2)T_{n+1}'+(n+1)(XT_{n+1}-T_n)=0 \label{eq_4} \\
		&(n+1)T_{n+1}=XU_n-(1-X^2)U_n'. \label{eq_5} 
	\end{align}
\end{lem}

\begin{proof} The first three equalities are showed by induction. All these ones are clear if $n=0$ and $n=1$ so we only make the induction step. In each case, we assume that the property is true for $n$ and $n+1$. 

For \eqref{eq_1}, we can write
		\begin{align*}
		T_{n+3}&=2XT_{n+2}-T_{n+1}=2X(U_{n+2}-XU_{n+1})-(U_{n+1}-XU_n) \\
		&=2XU_{n+2}-U_{n+1}-X(2XU_{n-1}-U_n)=U_{n+3}-XU_{n+2}
		\end{align*} 
		
For \eqref{eq_2}, we use \eqref{eq_1}:
		\begin{align*}T_{n+3}'=(2XT_{n+2}-T_{n+1})'&=2T_{n+2}+2XT_{n+2}'-T_{n+1}' \\
		&=2(U_{n+2}-XU_{n+1})+2X(n+2)U_{n+1}-(n+1)U_{n} \\
		&=(n+1)(2XU_{n+1}-U_n)+2U_{n+2}=(n+3)U_{n+2}.
		\end{align*}
			
And, finally, for \eqref{eq_3},  
		\begin{align*}
		T_{n+4}&=2XT_{n+3}-T_{n+2}=2X(XT_{n+2}-(1-X^2)U_{n+1})-(XT_{n+1}-(1-X^2)U_n) \\
		&=X(2XT_{n+2}-T_{n+1})-(1-X^2)(2XU_{n+1}-U_n) \\
		&=XT_{n+3}-(1-X^2)U_{n+2}
		\end{align*}
		Thus, the first three equalities are proved by induction.

For the last two ones, direct calculations succeed, using previous results. For \eqref{eq_4}, we write:
		\begin{align*}
		(1-X^2)T_{n+1}'&=(1-X^2)(n+1)U_n=(n+1)(Xt_{n+1}-T_{n+2}) \\
		&=(n+1)(XT_{n+1}-(2XT_{n+1}-T_n))=(n+1)(T_n-XT_{n+1})
		\end{align*}

For \eqref{eq_5}, we differentiate \eqref{eq_3}. It yields: $T_{n+2}'=T_{n+1}+XT_{n+1}'+2XU_n-(1-X^2)U_n'$. But, by \eqref{eq_2} and \eqref{eq_1}, $T_{n+2}'=(n+2)U_{n+1}=(n+2)(T_{n+1}+XU_n)$ and, by \eqref{eq_2}, $T_{n+1}'=(n+1)U_n$. Then,
\[(n+2)(T_{n+1}+XU_n)=T_{n+1}+X(n+1)U_n+2XU_n-(1-X^2)U_n\]
and thus $(n+1)T_{n+1}=XU_n-(1-X^2)U_n'$.
\end{proof}

\begin{theo} Let $n\in\N^*$. All the roots of $T_n$ and $U_n$ are simple.
\end{theo}

\begin{proof} Let $w$ be a root of $T_n$. By Lemma \ref{lem_coprime}, $T_n$ and $T_{n+1}$ are coprime thus $w$ is not a root of $T_{n+1}$. By \eqref{eq_3}, $T_{n+1}(w)=(1-w^2)U_{n-1}(w)$ thus $U_{n-1}(w)\neq 0$. But, by \eqref{eq_2}, $T_n'=nU_{n-1}$ thus $T_n'(w)\neq 0$. Then, $w$ is a simple root of $T_n$.
	
	Let now $w$ be a root of $U_n$. By \eqref{eq_2}, $T_{n+1}'(w)=0$. Since all the roots of $T_{n+1}$ are simple, $T_{n+1}(w)\neq 0$. But, by \eqref{eq_5}, $(n+1)T_{n+1}(w)=(w^2-1)U_n'(w)$ thus $U_n'(w)\neq 0$ and $w$ is a simple root of $U_n$. 
\end{proof}

\section{The roots of Chebyshev polynomials are mostly irrational}

\begin{prop} \label{prop_diff_equa} For all $n\in\N$, $T_n$ satisfies the differential equation
	\begin{equation}
	(1-X^2)T_n''-XT_n'+n^2T_n=0
	\label{eq_diff_equa}
	\end{equation}
\end{prop}

\begin{proof} We use double induction. The base step is clear. Assume that $T_n$ and $T_{n+1}$ satisfy \eqref{eq_diff_equa} for a certain nonnegative integer $n$. Since $T_{n+2}=2XT_{n+1}-T_n$,we have
	\[T_{n+2}'=2T_{n+1}+2XT_{n+1}'-T_n' \qquad \text{and} \qquad T_{n+2}''=4T_{n+1}'+2XT_{n+1}''-T_n''\]
	thus
	\begin{align*}
	(1-X^2)&T_{n+2}''-XT_{n+2}'+(n+2)^2T_{n+2} \\
	&=(1-X^2)(4T_{n+1}'+2XT_{n+1}''-T_n'')-X(2T_{n+1}+2XT_{n+1}'-T_n')+(n+2)^2(2XT_{n+1}-T_n) \\
	& 	=2X\underbrace{\left[(1-X^2)T_{n+1}''-XT_{n+1}'+(n+1)^2T_{n+1}\right]}_{=0}-\underbrace{\left[(1-X^2)T_n''-XT_n'+n^2T_n\right]}_{=0} \\
	&\hspace{3cm} +4(1-X^2)T_{n+1}'-2XT_{n+1}+[2(n+1)+1]2XT_{n+1}-(4n+4)T_n \\
	&=4\left[(1-X^2)T_{n+1}'+(n+1)(XT_{n+1}-T_n)\right].
	\end{align*}
Then, using \eqref{eq_4}, we conclude that $(1-X^2)T_{n+2}''-XT_{n+2}'+(n+2)^2T_{n+2}=0$. Hence, the property is true for $n+2$ and Proposition \ref{prop_diff_equa} is proved by double induction.
\end{proof}

\begin{prop} 	\label{prop_polynomial_expansions} For all $n\in\N^*$,
	\[T_n=\frac{n}{2}\sum_{k=0}^{\lfloor \frac{n}{2}\rfloor} \frac{(-1)^k2^{n-2k}}{n-k}\binom{n-k}{k} X^{n-2k} \qquad \text{and} \qquad U_n=\sum_{k=0}^{\lfloor \frac{n}{2}\rfloor} (-1)^k 2^{n-2k} \binom{n-k}{k} X^{n-2k}\]
\end{prop}

\begin{proof} Here, we follow \cite[p. 24-25]{Sny66}. Let $n$ be an positive integer and put $m=\lfloor \frac{n}{2} \rfloor$. Since $T_n$ is a polynomial in $\Z[X]$ of degree $n$ which has the same parity as $n$, we can write 
	\[T_n=\sum_{k=0}^{m} c_k X^{n-2k}\]
	where $c_k\in\Z$ for all $k\in\llbracket 0, m \rrbracket$.
	
	By substituting this expression in \eqref{eq_diff_equa}, one gets 
	\[(1-X^2)\sum_{k=0}^{m} (n-2k)(n-2k-1)c_k X^{n-2k-2} - X \sum_{k=0}^{m} (n-2k)c_k X^{n-2k-1} + n^2\sum_{k=0}^{m} c_k X^{n-2k} = 0\]
	i.e., by putting $c_{-1}=0$,
	\[\sum_{k=0}^{m} ((n-2k+2)(n-2k+1)c_{k-1}-(n-2k)(n-2k-1)c_k - (n-2k)c_k + n^2c_k) X^{n-2k} = 0\]
	One deduces that, for all $k\in\llbracket 0, m \rrbracket$,
	\[(n-2k+2)(n-2k+1)c_{k-1}+(n^2-(n-2k)^2)c_k=0\]
	and thus, for all $k\in\llbracket 1, m\rrbracket$
	\[c_k=-\frac{(n-2k+1)(n-2k+2)}{4k(n-k)}c_{k-1}.\]
	Hence, for all $k\in\llbracket 0, m \rrbracket$,
	\[c_k=\frac{(-1)^k}{4^k} c_0 \prod_{j=1}^{k} \dfrac{(n-2j+1)(n-2j+2)}{j(n-j)}=\frac{(-1)^k}{4^k}c_0\frac{n!(n-k-1)!}{k!(n-2k)!(n-1)!}.\]
	By an obvious double induction, we get that $c_0=2^{n-1}$ thus, for all $k\in\llbracket 0, m \rrbracket$,
	\[c_k=(-1)^k 2^{d-1-2k}\frac{d(d-k-1)!}{k!(d-2k)!}=\frac{d}{2} \frac{(-1)^k2^{d-2k}}{d-k} \binom{d-k}{k}\]
	as announced. Moreover, by \eqref{eq_2},
	\begin{align*}
	U_n&=\frac{1}{n+1}T_{n+1}'=\frac{1}{n+1}\cdot \frac{n+1}{2} \sum_{k=0}^{\lfloor \frac{n+1}{2}\rfloor} \frac{(-1)^k2^{n+1-2k}}{n+1-k}\binom{n+1-k}{k}(n+1-2k)X^{n-2k} \\
	&=\sum_{k=0}^{\lfloor \frac{n}{2}\rfloor} (-1)^k2^{n-2k}\frac{(n-k)!}{k!(n-2k)!}X^{n-2k}=\sum_{k=0}^{\lfloor \frac{n}{2}\rfloor} (-1)^k2^{n-2k} \binom{n-k}{k} X^{n-2k}
	\end{align*}
	because $\lfloor \frac{n+1}{2} \rfloor = \lfloor \frac{n}{2} \rfloor$ if $n$ is even and $n+1-2\lfloor \frac{n+1}{2} \rfloor=0$ if $n$ is odd.
\end{proof}

\begin{theo} \label{theo_irrationality} Let $n$ be a positive integer.
	\begin{enumerate}
		\item The roots of $T_n$ are all irrational if $n$ is even and $0$ is the only rational root of $T_n$ is $n$ is odd.
		\item The only possible rational roots of $U_n$ are $0$, $\frac{1}{2}$ and $-\frac{1}{2}$. Moreover, $0$ is a root of $U_n$ if and only if $n$ is odd and $\frac{1}{2}$ and $-\frac{1}{2}$ are roots of $U_n$ if and only if $n \equiv 2 \pmod{3}$.
	\end{enumerate}
\end{theo}

\begin{proof} Let us put $m=\lfloor \frac{n}{2} \rfloor$.
	 \begin{enumerate}
		\item Assume that $\alpha$ is a rational root of $T_n$. Thus, $2\alpha$ is a root of $2T_n(\frac{X}{2})$. But, according with Proposition \ref{prop_polynomial_expansions},
		\[2T_n\left(\frac{X}{2}\right)=\sum_{k=0}^{\lfloor \frac{n}{2}\rfloor} (-1)^k \frac{n}{n-k}\binom{n-k}{k} X^{n-2k}.\] 
		Note that, for all $k\in\llbracket 0, m\rrbracket$, $\frac{n}{n-k}\binom{n-k}{k}=\binom{n-k}{k}+\frac{k}{n-k}\binom{n-k}{k}=\binom{n-k}{k}+\binom{n-k-1}{k-1}$ so $2T_n(\frac{X}{2})$ has integral coefficients. Since its leading term is $1$, we deduce from a classic result about algebraic integers (see, for example, \cite[Corollary 6.14, p. 308]{NZM91}) that $2\alpha$ is an integer. Moreover, by Theorem \ref{theo_roots_intervalle}, $\alpha$ lies in $(-1,1)$ thus $\alpha\in\left\{0, -\frac{1}{2}, \frac{1}{2}\right\}$.
		
		Let us recall that by Lemma \ref{lem_linear_rec}, $T_n(\alpha)=\frac{r^{2d}+1}{2r^d}$ where $r$ is a root of $X^2-2\alpha X + 1$. If $\alpha=0$, we can take $r=\mathrm{i}$ and thus $T_n(\alpha)=-\frac{(-1)^n+1}{2}$ and we conclude that $0$ is a root of $T_n$ if and only if $n$ is odd. If $\alpha=-\frac{1}{2}$ then we can take $r=j:=-\frac{1}{2}+\mathrm{i}\frac{\sqrt{3}}{2}$. Therefore, $T_n(\alpha)=0$ if and only if $j^{2n}=-1$. But $j^3-1=(j-1)(j^2+r+1)=0$ thus $j^3=1$. It follows that the sequence $(j^d)_{d\in\N}$ is $3-$periodic and since $j^0=1$, $j=-\frac{1}{2}+\mathrm{i}\frac{\sqrt{3}}{2}$ and $j^2=-\frac{1}{2}-\mathrm{i}\frac{\sqrt{3}}{2}$ we conclude that $j^{2n} \neq -1$, i.e., $-\frac{1}{2}$ is not a root of $T_n$. Finally, due to the parity of $T_n$, $\frac{1}{2}$ is not a root of $T_n$ either.
		\item Assume that $\alpha$ is a rational root of $U_n$. By considering $U_n(\frac{X}{2})$, one proves similarly that $\alpha\in\{0,-\frac{1}{2},\frac{1}{2}\}$. The case $\alpha=0$ leads to the same conclusion. Nonetheless, the case $\alpha=-\frac{1}{2}$ is a bit different. Indeed, one has $U_n(\alpha)=\frac{j^{2n+2}-1}{j^n(j^2-1)}$ and, according to the periodic values of the sequence $(j^d)$, one sees that $U_n(\alpha)=0$ if and only if $3$ divides $2n+2$ that is to say $n \equiv 2 \pmod{3}$. The same argument of parity concludes the proof.
	\end{enumerate}
\end{proof}

The previous proof depends largely on the result of section 2 and especially on Theorem \ref{theo_roots_intervalle} which ensures that $s\in\{0,-1,1\}$. In fact, there is a way to derive the irrationality of non-zero roots of $T_n$ by using only the result of the present section. The proof is based on the shifted Chebyshev polynomials.

\begin{defi} For all nonnegative integer $n$, one defines the shifted Chebyshev polynomials (of first kind) $T_n^*$ by:
	\[T_n^*=T_n(2X-1).\]
\end{defi}

\begin{prop} For all nonnegative integer $n$, $T_n^*(X^2)=T_{2n}$.
	\label{prop_like_shifted_polynomials}
\end{prop}

\begin{proof} One uses double induction. The result is clear if $n=0$ and $n=1$.  Assume the relation is true for $n$ and $n+1$. Then,  
	\begin{align*}
	T_{n+2}^*(X^2)&=T_{n+2}(2X^2-1)=2(2X^2-1)T_{n+1}(2X^2-1)-T_n(2X^2-1)\\
	&=2(2X^2-1)T_{n+1}^*(X^2)-T_{n}^*(X^2) = 2(2X^2-1)T_{2n+2}-T_{2n} \\
	&=4X^2T_{2n+2}-(T_{2n+2}+T_{2n})-T_{2n+2}=4X^2T_{2n+2}-2XT_{2n+1}-T_{2n+2} \\
	&=2X(2XT_{2n+2}-T_{2n+1})-T_{2n+2} =2XT_{2n+3}-T_{2n+2} \\
	&=T_{2n+4}.
	\end{align*}
	Hence, the relation is true for $n+2$ and Proposition \ref{prop_like_shifted_polynomials} is proved by induction.
\end{proof}

\begin{theo} Let $n$ be an positive integer. The only possible rational root of $T_n^*$ is $\frac{1}{2}$. 
\end{theo}

\begin{proof} It follows from Propositions \ref{prop_polynomial_expansions} and \ref{prop_like_shifted_polynomials} that 
	\[T_n^*(X^2)=n\sum_{k=0}^{n} \frac{(-1)^k2^{2n-2k}}{2n-k}\binom{2n-k}{k}X^{2n-2k}=n\sum_{k=0}^{n} \frac{(-1)^{n-k}2^{2k}}{n+k}\binom{n+k}{2k}(X^2)^{k}\]
	thus
	\[2T_n^*\left(\frac{X}{4}\right)=\sum_{k=0}^{n} (-1)^{n-k}\frac{2n}{n+k}\binom{n+k}{2k}X^{k} = \sum_{k=0}^{n} (-1)^{n-k}\left[2\binom{n+k}{n-k}-\binom{n+k-1}{n-k}\right]X^{k}.\]
	Hence, $2T_n^*(\frac{X}{4})$ has integral coefficients. Assume that $\alpha$ is a rational root of $T_n^*$. Note that the terms of $T_n^*$ alternate in sign and the constant term of $T_n^*$ is non-zero so $\alpha >0$. Since $2T_n^*(\frac{X}{4})$ is a monic polynomial and has constant term $2$, the same argument that for $T_n$ yields that $4\alpha$ is an integer dividing $2$ so $\alpha\in\{\frac{1}{2}, \frac{1}{4}\}$. Moreover, due to the parity of $T_n$, $1-\alpha$ is also a rational root of $T_n^*$ so we conclude that the only possible rational root of $T_n^*$ is $\frac{1}{2}$.  
\end{proof}

As an immediate corollary of the previous theorem, we deduce that $T_n$ has no rational root if $n$ is even and $0$ is the only rational root of $T_n$ is $n$ is odd.

\bibliographystyle{amsalpha}
\bibliography{biblio}

\end{document}